\documentclass[12pt,twoside]{amsart}
\usepackage{amssymb}

\usepackage[latin1]{inputenc}
\usepackage[T1]{fontenc}

\usepackage[all]{xy}
\nonstopmode

\headsep=0.5cm \numberwithin{equation}{section}
\font\tengothic=eufm10 scaled\magstep 1 \font\sevengothic=eufm7
scaled\magstep 1
\newfam\gothicfam
       \textfont\gothicfam=\tengothic
       \scriptfont\gothicfam=\sevengothic
\def\goth#1{{\fam\gothicfam #1}}

\newtheorem{theorem}{Theorem}[section]

\newtheorem{proposition}[theorem]{Proposition}
\newtheorem{corollary}[theorem]{Corollary}

\theoremstyle{definition}
\newtheorem{definition}[theorem]{Definition} 
\newtheorem{remark}[theorem]{Remark}
\newtheorem{example}[theorem]{Example}

\newcommand{\codim}{\operatorname{codim}}
\newcommand{\coker}{\operatorname{coker}}

\newcommand{\deter}{\operatorname{det}}

\newcommand{\degr}{\operatorname{deg}}

\newcommand{\Hom}{\operatorname{Hom}}

\newcommand{\ext}{\operatorname{ext}}
\newcommand{\Ext}{\operatorname{Ext}}

\newcommand{\Spec}{\operatorname{Spec}}

\newcommand{\depth}{\operatorname{depth}}

\newcommand{\im}{\operatorname{im}}

\newcommand{\HH}{H}

\newcommand{\Hi}{\operatorname{Hilb}}
\newcommand{\Gr}{\operatorname{GradAlg}}

\newcommand{\Proj}{\operatorname{Proj}}

\newcommand\sO{{\mathcal O}}

\newcommand{\cA}{{\mathcal A}}

\newcommand{\cU}{{\mathcal U}}

\newcommand{\cO}{{\mathcal O}}

\newcommand{\cN}{{\mathcal N}}

\newcommand {\QQ}{\mathbb{Q}}

\newcommand {\PP}{\mathbb{P}}

\newcommand {\ra}{\longrightarrow}

\begin{document}
\title[]{Ideals generated by submaximal minors.}

\author[Jan O.\ Kleppe, Rosa M.\ Mir\'o-Roig]{Jan O.\ kleppe, Rosa M.\
Mir\'o-Roig$^{*}$}
\address{Faculty of Engineering,
         Oslo University College,
         Postboks 4, St. Olavs Plass, N-0130 Oslo,
         Norway}
\email{JanOddvar.Kleppe@iu.hio.no}
\address{Facultat de Matem\`atiques,
Departament d'Algebra i Geometria, Gran Via de les Corts Catalanes
585, 08007 Barcelona, SPAIN } \email{miro@ub.edu}

\date{\today}
\thanks{$^*$ Partially supported by MTM2007-61104.}



\begin{abstract}
The goal of this paper is to study irreducible families
$W^{t-1}_{t,t}(\underline{b};\underline{a})$ of codimension 4,
arithmetically Gorenstein schemes $X\subset \PP^n$ defined by the
submaximal minors of a $t\times t$ homogeneous matrix $\cA $ with
entries homogeneous forms of degree $a_j-b_i$. Under some
numerical assumption on $a_j$ and $b_i$ we prove that the closure
of $W^{t-1}_{t,t}(\underline{b};\underline{a})$ is an irreducible
component of $\Hi ^{p(x)}(\PP^{n})$, we show that $\Hi
^{p(x)}(\PP^{n})$ is generically smooth along
$W^{t-1}_{t,t}(\underline{b};\underline{a})$ and we compute the
dimension of $W^{t-1}_{t,t}(\underline{b};\underline{a})$ in terms
of $a_j$ and $b_i$. To achieve these results we first prove that
$X$ is determined by a regular section of  ${\mathcal
I}_Y/{\mathcal I}_Y^2(s)$ where $s=\deg(\det(\cA))$ and $Y\subset
\PP^n$ is a codimension 2, arithmetically Cohen-Macaulay scheme
defined by the maximal minors of the matrix  obtained deleting a
suitable row of $\cA $.
\end{abstract}


\maketitle

\tableofcontents


  \section{Introduction} \label{intro}

In this paper we deal with determinantal schemes. A scheme
$X\subset \PP^{n}$ of codimension $c$ is called {\em
determinantal} if its homogeneous saturated ideal can be generated
by the $r \times r$ minors of a homogeneous $p \times q$ matrix
with $c=(p-r+1)(q-r+1)$. When $r=min(p,q)$ we say that $X$ is {\em
standard determinantal}. Given integers $r\le p \le q$, $a_1\le
a_2\le ...\le a_{p}$ and $b_1\le b_2\le ...\le b_q$ we denote by
$W^r_{p,q}(\underline{b};\underline{a})\subset \Hi
^{p(x)}(\PP^{n})$ the locus of  determinantal schemes $X\subset
\PP^{n}$ of codimension $c=(p-r+1)(q-r+1)$ defined by the $r\times
r$ minors of a $p\times q$ matrix
$(f_{ji})^{i=1,...,q}_{j=1,...,p}$ where $f_{ji}\in
k[x_0,x_1,...,x_{n}]$ is a homogeneous polynomial of degree
$a_j-b_i$.

The study of determinantal schemes has received considerable
attention in the literature (See, for instance, \cite{b-v},
\cite{e-h}, \cite{e-n} and \cite{rmmr}). Some classical schemes
that can be constructed in this way are the Segre varieties,
rational normal scrolls, and the Veronese varieties. The main goal
of this paper is to contribute to the classification of
determinantal schemes and we  address  in the case $p=q=t$,
$r=t-1$ the following  three fundamental problems:
\begin{itemize} \item[(1)] To determine the dimension of
$W^r_{p,q}(\underline{b};\underline{a})$ in terms of $a_j$ and
$b_i$, \item[(2)] To determine whether the closure of
$W^r_{p,q}(\underline{b};\underline{a})$ is an irreducible
component of $\Hi ^{p(x)}(\PP^{n})$, and \item[(3)] To determine
when $\Hi ^{p(x)}(\PP^{n})$ is generically smooth along
$W^r_{p,q}(\underline{b};\underline{a})$.
\end{itemize}

 The first
important contribution to this problem is due to G. Ellingsrud
\cite{elli}; in 1975, he proved that every arithmetically
Cohen-Macaulay, codimension 2 closed subscheme $X$ of $\PP^{n}$ is
unobstructed (i.e. the corresponding point in the Hilbert scheme
$\Hi ^{p(x)} (\PP^{n})$ is smooth) provided $n\ge 3$ and he also
computed the dimension of the Hilbert scheme at $(X)$. Recall also
that the homogeneous ideal of an arithmetically Cohen-Macaulay,
codimension 2 closed subscheme $X$ of $\PP^{n}$ is given by the
maximal minors of a $(t-1)\times t$ homogeneous matrix, the
Hilbert-Burch matrix. That is, such a scheme is standard
determinantal. The purpose of this work is to extend Ellingsrud's
Theorem, viewed as a statement on standard determinantal schemes
of codimension 2, to arbitrary determinantal schemes. The case of
 codimension 3 standard determinantal schemes,
 was mainly solved in \cite{KMMNP}; Proposition 1.12; and the case of standard
 determinantal schemes of arbitrary codimension was studied and partially solved
  in \cite{KM}.
In \cite{KM1}, we treated the case of codimension 3 determinantal
schemes $X\subset \PP^n$ defined by the submaximal minors of a
symmetric homogeneous matrix. In our opinion, it is difficult to
solve the above three questions in full generality and, in this
paper, we will focus our attention to the first unsolved case,
namely, we will deal with codimension 4 determinantal schemes
$X\subset \PP^n$, $n\ge 5$, defined by the submaximal minors of a
homogeneous square matrix. As in \cite{KMMNP}, \cite{KM} and
\cite{KM1}, we prove our results by considering the smoothness of
the Hilbert flag scheme of pairs or, more generally, the Hilbert
flag scheme of chains of closed subschemes obtained by deleting
suitable rows, and its natural projections into the usual Hilbert
scheme. We wonder if a similar strategy could facilitate the study
of the general case.

\vskip 2mm Next we outline the structure of the paper. In section 2, we recall
the basic facts on local cohomology and deformation theory needed in the
sequel. In section 3, we describe the deformations of the codimension 4
arithmetically Gorenstein schemes $X\subset \PP^n$ defined as the degeneracy
locus of a regular section of the twisted conormal sheaf ${\mathcal
  I}_Y/{\mathcal I}_Y^2(s)$ of a codimension 2, arithmetically Cohen-Macaulay
scheme $Y \subset \PP^n$ of dimension $\ge 3$. Section 4 is the
heart of the paper. In section 4, we determine the dimension of
$W^{t-1}_{t,t}(\underline{b};\underline{a})$ in terms of $b_i$ and
$a_j$ provided $a_i\ge b_{i+3}$ for $1\le i\le t-3$ (and $a_1\ge
b_t$ if $t\le 3$), $a_t>a_{t-1}+a_{t-2}-b_1$ and $\dim X\ge 1$. We
also prove that, under this numerical restriction, $\Hi
^{p(x)}(\PP^{n})$ is generically smooth along
$W^{t-1}_{t,t}(\underline{b};\underline{a})$ and the closure of
$W^{t-1}_{t,t}(\underline{b};\underline{a})$ is an irreducible
component of $\Hi ^{p(x)}(\PP^{n})$ (cf. Theorem \ref{mainthm1}).

\vskip 2mm The key point for proving our result is the fact that
any codimension 4, determinantal scheme $X\subset \PP^n$ defined
by the submaximal minors of a homogeneous square matrix $\cA $ is
arithmetically Gorenstein and determined by a regular section of
${\mathcal I}_Y/{\mathcal I}_Y^2(s)$ where $s=\deg(\det(\cA))$ and
$Y\subset \PP^n$ is a codimension 2, arithmetically Cohen-Macaulay
scheme defined by the maximal minors of the matrix $\cN $ obtained
deleting a suitable row of $\cA $ (cf. Proposition
\ref{sectionnormal}). Conversely, any codimension 4,
arithmetically Gorenstein scheme $X=\Proj(A)\subset \PP^n$ defined
by a regular section $\sigma $ of ${\mathcal I}_Y/{\mathcal
I}_Y^2(s)$ where $Y=\Proj(B)\subset \PP^n$ is a codimension 2,
arithmetically Cohen-Macaulay scheme, fits into an exact sequence
of the following type
$$0\longrightarrow K_B(n+1-2s)\longrightarrow N_B(-s)\stackrel{\sigma^*}{\longrightarrow} B\longrightarrow A\longrightarrow 0 $$
and it is determined by the submaximal minors of
 a $t\times t$ homogeneous
matrix $\cA $ obtained by adding a suitable row to the
Hilbert-Burch matrix of $Y$ (cf. Proposition \ref{sectionnormal}).
 In the last section, we include
some examples which illustrate that the numerical hypothesis in
Theorem \ref{mainthm1}, $a_t>a_{t-1}+a_{t-2}-b_1$, cannot be
avoided.

 \vskip 4mm

{\bf Notation.} Throughout this paper $k$ will be an algebraically
closed field $k$, $R=k[x_0, x_1, \dots ,x_n]$, $\goth m= (x_0,
\dots ,x_n)$ and $\PP^n=\Proj(R)$. As usual, the sheafification of
a graded $R$-module $M$ will be denoted by $\widetilde{M}$ and the
support of $M$ by $Supp(M)$.

\vskip 2mm Given a closed subscheme $X$ of $\PP^n$ of codimension
$c$, we denote by  ${\mathcal I}_X$ its ideal sheaf,  ${\mathcal
N}_X$ its normal sheaf and $ I(X)=H^0_{*}(\PP^n, {\mathcal I}_X)$
its saturated homogeneous ideal unless $X=\emptyset $, in which
case we let $I(X)=\goth m$ . If $X$ is equidimensional and
Cohen-Macaulay of codimension $c$, we set $\omega_X ={\mathcal
E}xt^c_{{\mathcal O}_{\PP^n}} ({\mathcal O}_X,{\mathcal
O}_{\PP^n})(-n-1)$ to be its canonical sheaf.

 In the sequel, for any graded quotient $A$ of $R$ of codimension $c$,
we let $I_A=\ker(R\twoheadrightarrow A)$, $N_A=\Hom_R(I_A,A)$ be
the normal module. If $A$ is Cohen-Macaulay of codimension $c$, we
let  $K_A=\Ext^c_R (A,R)(-n-1)$ be its canonical module. When we
write $X=\Proj(A)$, we let $A=R/I(X)$ and $K_X=K_A$. If $M$ is a
finitely generated graded $A$-module, let $\depth_{J}{M}$ denote
the length of a maximal $M$-sequence in a homogeneous ideal $J$
and let $\depth {M} = \depth_{\mathfrak m}{ M}$. If
$\Gamma_{J}(-)$ is the  functor of sections with support in $
\Spec (A/J)$, we denote by $\HH^i_{J}(-)$ the right derived
functor of $\Gamma_{J}(-)$.

\vskip 2mm Let $\Hi ^{p(x)}(\PP^n)$ be the Hilbert scheme
parameterizing closed subschemes $X$ of $\PP^n$ with Hilbert
polynomial $p(x)\in \QQ[x]$ (cf. \cite{G}). By abuse of notation
we will write $(X)\in \Hi ^{p(x)}(\PP^n)$ for the $k$-point which
corresponds to a closed subscheme $X\subset \PP^n$. The Hilbert
polynomial of $X$ is sometimes denoted by $p_X$. By definition $X$
is called unobstructed if $\Hi ^{p(x)}(\PP^n)$ is smooth at $(X)$.

The pullback of the universal family on $ \Hi ^p(\PP^n)$ via a
morphism $\psi:W\ra   \Hi ^p(\PP^n)$ yields a flat family over
$W$, and we will write $(X)\in W$ for a member of that family as
well. Suppose that $W$ is irreducible. Then, by definition a
general $(X)\in W$ has a certain property if there is a non-empty
open subset $U$ of $W$ such that all members of $U$ have this
property. Moreover, we say that $(X)$ is general in $W$ if it
belongs to a sufficiently small open subset $U$ of $W$ (small
enough to satisfy all the openness properties that we want to
require).

Finally we let $D=D(p_X,p_Y)$ be the Hilbert flag scheme
parameterizing pairs of closed subschemes $(X' \subset Y')$ of
$\PP^{n}$ with Hilbert polynomials $p_{X'}=p_X$ and $p_{Y'}=p_Y$,
respectively.


\section{Preliminaries}

For convenience of the reader we include in this section the
background and basic results on local cohomology and deformation
theory needed in the sequel.

\subsection{Local cohomology}

Let $B=R/I_B$ be a graded quotient of the polynomial ring $R$, let
$M$ and $N$ be finitely generated graded $B$-modules and let $J
\subset B$ be an ideal. We say that $0\ne M$ is a Cohen-Macaulay
(resp. maximal Cohen-Macaulay) $B$-module if  $\depth {M}= \dim
{M}$ (resp. $\depth {M}= \dim {B}$), or equivalently,
$\HH_{\mathfrak m}^i({ M})=0$  for all $i\ne \dim {M}$ (resp.
$i<\dim {B}$) since $\depth_{J}{M} \geq r $ is equivalent to $
\HH^i_{J}({M})=0 {\ \rm for \ } i<r$. If $B$ is Cohen-Macaulay, we
know by Gorenstein duality that the $v$-graded piece of
$\HH_{\mathfrak m}^i({ M})$ satisfies
$$_{v}\HH_{\mathfrak{m}}^{i}(M) \simeq \
_{-v}\!\Ext_B^{\dim B-i}(M,K_B)^{\vee}.$$ Let $Z$ be closed in
$Y:=\Proj ({B})$ and let $U=Y-Z$. Then we have an exact sequence
$$0 \rightarrow \HH^0_{I(Z)}(M) \rightarrow M \rightarrow
\HH^0_*(U,\widetilde{M}) \rightarrow \HH^1_{I(Z)}(M) \rightarrow
0$$ and isomorphisms $\HH^i_{{I}(Z)}({M}) \simeq
\HH^{i-1}_*(U,\widetilde {M})\ {\rm for} \ i \geq 2$ where as
usual we write $\HH^{i}_*(U,\widetilde {M})=\oplus
_t\HH^{j}(U,\widetilde {M}(t))$. More generally, if
$\depth_{I(Z)}N \geq i+1$ there is an exact sequence
\begin{equation} \label{globalext}
 \small
 _0\!\Ext_B^i(M,N) \hookrightarrow \Ext_{\sO_U}^i(\widetilde M
 \arrowvert_U,\widetilde N \arrowvert_U)  \rightarrow \ _0\!\Hom_B(M,
 \HH^{i+1}_{I(Z)}(N)) \rightarrow \  _0\!\Ext_B^{i+1}(M,N)
  \rightarrow
\end{equation}
by \cite{SGA2}; exp. VI, where the middle form comes from a
spectral sequence also treated in \cite{SGA2}.

\subsection{Basic deformation theory} To use deformation theory, we will need to consider the
(co)homology groups of algebras $\HH_2(R,B,B)$ and $\HH^2(R,B,B)$.
Let us recall their definition.  We consider
\begin{equation} \label{res}
 ...\rightarrow  F_2:= \oplus _{j=1}^{\mu_2}R(-n_{2,j}) \rightarrow  F_1:=
  \oplus_{i=1}^{\mu } R(-n_{1,i}) \rightarrow R \rightarrow B \rightarrow 0  \
\end{equation}
 a minimal graded free $R$-resolution of $B$ and let $\HH_1=\HH_1(I_B)$ be the
 first
Koszul homology built on a set of minimal generators of $I_B$.
Then we may take the exact sequence
\begin{equation} \label{alghom}
   0 \rightarrow \HH_2(R,B,B) \rightarrow  {\HH}_1 \rightarrow
  {F}_{1} \otimes_{R} {B} \rightarrow {I_B}/{I_B}^2 \rightarrow 0
\end{equation}
as a definition of the second algebra homology $\HH_2(R,B,B)$ (cf.
\cite{VAS}), and the dual sequence, $$ \rightarrow \
_v\!\Hom_B(F_1 \otimes B,B)\rightarrow \ _v\!\Hom_B(\HH_1,B)
\rightarrow \ _v\HH^2(R,B,B) \rightarrow 0,$$ as a definition of
graded second algebra cohomology $\HH^2(R,B,B)$. If $B$ is
generically a complete intersection, then it is well known that $\
\Ext_B^1(I_B/I_B^2,B) \simeq \ \HH^2(R,B,B)$ (\cite{AND};
Proposition 16.1). We also know that $ \HH^0(Y,\cN_Y)$ is the
tangent space of $ \Hi^{ p(x)} (\PP^n)$ in general, while $
\HH^1(Y,\cN_Y)$ contains the obstructions of deforming $Y \subset
\PP^n$ in the case $Y$ is locally a complete intersection (l.c.i.)
(cf. \cite{G}). If $\ _0\!\Hom_R (I_B,\HH_{\mathfrak m}^1(B)) = 0$
(e.g. $\depth_{\mathfrak m}B \geq 2$), we have by
\eqref{globalext} that $\ _0\!\Hom_B(I_{B}/I_B^2,B) \simeq \
\HH^0(Y,\cN_Y)\ $ and $\ _0\HH^2(R,B,B) \hookrightarrow
\HH^1(Y,\cN_Y)$ is injective in the l.c.i. case, and that $\
_0\HH^2(R,B,B)$ contains the obstructions of deforming $Y \subset
\PP^n$ (\cite{K79}; Remark 3.7). Thus $\ _0\HH^2(R,B,B)=0$
suffices for the unobstructedness of a l.c.i. arithmetically
Cohen-Macaulay subscheme $Y$ of $\PP^n$ of $\dim Y \geq 1$ (for
this conclusion we may even entirely skip ``l.c.i.'' by slightly
extending the argument, as done in \cite{K79}).

\subsection{Useful exact sequences} In the last part of this section, we
collect some exact sequences frequently used in this paper, in the
case that $B=R/I_B$ is a generically complete intersection
codimension two CM quotient of $R$. First, applying $\Hom_R(-,R)$
to the minimal graded free $R$-resolution of $B$
\begin{equation} \label{rescodim2}
 0 \rightarrow  F_2:= \oplus _{j=1}^{\mu -1}R(-n_{2,j}) \rightarrow  F_1:=
  \oplus_{i=1}^{\mu } R(-n_{1,i}) \rightarrow R \rightarrow B \rightarrow 0
  \,
\end{equation}
we get a minimal graded free $R$-resolution of $K_B$:
 \begin{equation} \label{Kcodim2}
  0 \rightarrow R \rightarrow  \oplus R(n_{1,i}) \rightarrow
  \oplus R(n_{2,j}) \rightarrow K_B(n+1) \rightarrow 0.
\end{equation}
If we apply $\Hom(-,B)$ to \eqref{Kcodim2} we get the exactness to
the left in the exact sequence
\begin{equation}  \label{K*codim2}
 0 \rightarrow K_B(n+1)^* \rightarrow \oplus
  B(-n_{2,j}) \rightarrow \oplus B(-n_{1,i})\rightarrow I_B/I_B^2 \rightarrow 0
\end{equation}
which splits into two short exact sequences ``via $ \oplus
B(-n_{2,j}) \twoheadrightarrow H_1 \hookrightarrow \oplus
B(-n_{1,i})$'', one of which is \eqref{alghom} with $ \HH_2(R,B,B)
=0$. Indeed since $H_1$ is Cohen-Macaulay by  \cite{AH}, we get $
\HH_2(R,B,B) =0$ by (\ref{alghom}). Moreover since $
\Ext_R^1(I_B,I_B) \simeq N_B$ we showed in \cite{K03}; pg. 788
that there is an exact sequence of the form
 \begin{equation} \label{NBcodim2}
  0 \rightarrow F_1^* \otimes_R F_2 \rightarrow  (( F_1^* \otimes_R F_1)
  \oplus (F_2^* \otimes_R F_2))/R \rightarrow
  F_2^* \otimes_R F_1 \rightarrow N_B \rightarrow 0 \
\end{equation}
where $F_{i}^*=\Hom_R(F_{i},R)$. Indeed this sequence is deduced
from the exact sequence
$$ 0 \rightarrow R \rightarrow \oplus I_B(n_{1,i}) \rightarrow
\oplus I_B(n_{2,j}) \rightarrow N_B \rightarrow 0$$ which we get by applying
$\Hom_R(-,I_B)$ to \eqref{rescodim2}, (cf. \cite{K03};\! (26)). Similarly
applying $\Hom_R(-,I_B/I_B^2)$ to \eqref{rescodim2} and noting that
$\Hom_R(I_B,I_B/I_B^2) \simeq \Hom_B(I_B/I_B^2,I_B/I_B^2)$ we get the exact
sequence
  \begin{equation} \label{NB2codim2}
  0 \rightarrow \Hom_B(I_B/I_B^2,I_B/I_B^2) \rightarrow \oplus
  I_B/I_B^2(n_{1,i}) \rightarrow \oplus I_B/I_B^2 (n_{2,j}) \rightarrow N_B
  \rightarrow 0.
\end{equation}
Finally we recall the following frequently used exact sequence (cf. \cite{VAS})
 \begin{equation} \label{I2codim2}
  0 \rightarrow \wedge^2( \oplus R(-n_{2,j})) \rightarrow  (\oplus R(-n_{1,i}))
  \otimes ( \oplus R(-n_{2,j})) \rightarrow
  S^2( \oplus R(-n_{1,i})) \rightarrow I_B^2 \rightarrow 0 \ .
\end{equation}

\section{Deformations of quotients of regular sections}

In \cite{K} the first author studied deformations of a scheme $X:=\Proj(A)$
defined as the degeneracy locus of a regular section of a ``nice'' sheaf
${\widetilde { M}}$ on an arithmetically Cohen-Macaulay (ACM) scheme
$Y=\Proj(B)$. Recall that if we take a regular section of the anticanonical
sheaf $\widetilde K_B^*(s)$ and $Y$ is a l.c.i. of positive dimension, then
we get an exact sequence
\begin{equation*} \label{kb}
 0 \rightarrow K_B(-s) \rightarrow B \rightarrow A \rightarrow 0,
\end{equation*}
in which $A$ is Gorenstein. Indeed the mapping cone construction leads to a
resolution of $A$ from which we easily see that $A$ is Gorenstein. In
\cite{KP}, we generalized this way of constructing Gorenstein algebras to
sheaves of higher rank and,  in \cite{K}, we studied the deformations of this
``construction'', notable in the rank two case which we now recall.

Let $M$ be a maximal Cohen-Macaulay $B$-module of rank $r=2$ such that
$\widetilde {M}\arrowvert_U $ is locally free and $\wedge^2{\widetilde {
    M}}\arrowvert_U \simeq \widetilde {K_B}(t)\arrowvert_U$ in an open set
$U:=Y-Z$ of $Y$ satisfying $\depth_{I(Z)}B \geq 2$. Then a regular section
$\sigma$ of $\widetilde{M}^*(s)\arrowvert_U$ defines an arithmetically
Gorenstein scheme $X=\Proj(A)$ given by the exact sequence
\begin{equation} \label{introMs} 0 \rightarrow K_B(t-2s)\rightarrow M(-s)
  \stackrel{\sigma^*}{\rightarrow} B \rightarrow A \rightarrow 0 \
\end{equation}
and $M \simeq \Hom_B(M,K_B(t))$ (\cite{KP}; Theorem 8). In this
paper we consider and further develop the case where $M = N_B$ and
$\dim B = n-1$ ($n+1=\dim R$, $n\ge 5$). By \cite{KP}; Proposition
13, $N_B$ is a maximal Cohen-Macaulay $B$-module and we have the
exact sequence
\begin{equation} \label{Ns} 0 \rightarrow K_B(n+1-2s)\rightarrow N_B(-s)
  \rightarrow I_{A/B} \rightarrow 0 \ \hspace{5pt}, \hspace{5pt} {\rm where} \
  \ I_{A/B}:=\ker(B \rightarrow A).
\end{equation}

\begin{example}
  Set $R=k[x_0,\cdots ,x_5]$ and let $B=R/I_B$ be a
  codimension two quotient with minimal resolution
\begin{equation*} \label{I}
 0 \rightarrow R(-3)^2 \rightarrow  R(-2)^3  \rightarrow R
 \rightarrow B \rightarrow 0  \
\end{equation*}
and suppose $Y= \Proj(B)$ is a l.c.i in $ \mathbb P^5$. Let $A$ be given by a
regular section of $\widetilde {I_B/I_B^2}(s)$, $s \ge
3$.
Thanks to the exact sequences \eqref{Kcodim2} and \eqref{NBcodim2} and the mapping cone
construction applied to both \eqref{Ns} and $ 0 \to I_{A/B} \to B \rightarrow
A \to 0 $, we get the following resolution of the Gorenstein algebra $A$,
\begin{equation*} \label{NBA}
 \begin {aligned}
  & 0 \rightarrow R(-2s) \rightarrow  R(2-2s)^3 \oplus R(-1-s)^6 \\
 & \rightarrow R(3-2s)^2 \oplus R(-s)^{12} \oplus R(-3)^2 \rightarrow
 R(1-s)^6 \oplus R(-2)^3  \rightarrow R \rightarrow A
 \rightarrow 0 \ .
 \end{aligned}
\end{equation*}
Indeed $X = \Proj(A)$ is an arithmetically Gorenstein curve of
degree $d = 3s^2-10s+9$ and arithmetic genus $g =1+d(s-3)$ in $
\mathbb P^5$ (see \cite{K}; Example 43).
\end{example}

With $M$ and $A$ as above, it turns out that \cite{K}; Theorem 1
and Theorem 25 describe the deformations space, $\Gr(R)$, of the
graded quotient $A$ and computes the dimension of $\Gr(R)$ in
terms of a number $\delta:=\delta(K_B)_{t-2s} - \delta(M)_{-s}$
where
\begin{equation} \label{delta} \delta(N)_{v}: =\ _{v}\!\hom_B(I_B/I_B^2,N)-\
  _{v}\!\ext_B^1(I_B/I_B^2,N).
\end{equation}
Here we have used small letters for the $k$-dimension of $_v\!\Ext_B^i(-,-)$
and of similar groups. If we suppose $M=N_B$, $\depth_{I(Z)} B \geq 4$ and
$char(k) \neq 2$, then the conditions of the A (resp.\! B)-part of \cite{K}; Theorem 25
are satisfied provided $\ _{0}\!\Ext^2_B(N_B,N_B)=0$ (resp. $\
_{-s}\!\Ext^1_B(I_B/I_B^2,N_B)=0$). In both cases $X$ is unobstructed and
 \begin{equation} \label{dimHilb}
 \dim_{(X)} \Hi^{p(x)}(\PP^{n})= \ \dim (N_B)_0 + \ \dim (I_B/I_B^2)_s - \
 _{0}\!\hom_B(I_B/I_B^2,I_B/I_B^2)+ \ \dim(K_B)_{t-2s} +\delta
\end{equation}
where $t=n+1$ (see \cite{K}; Corollary 41 and its proof, and
Remark 42). Using the exact sequence \eqref{NBcodim2} we get $\
_{-s}\!\Ext^1_B(I_B/I_B^2,N_B)=0$ for $s > 2 \max n_{2,j}- \min
n_{1,i}$ which led to Corollary 41 of \cite{K} which we slightly
generalize in Corollary~\ref{introcorNB} (i) below. The A)-part
was considered in Remark 42 of \cite{K}. By the proof of Theorem
25 of \cite{K} we may replace the vanishing of $\
_{0}\!\Ext^2_B(N_B,N_B)$ by the vanishing of the subgroup $\
_{t}\!\Ext_B^2(S^2(I_{A/B}(s)),K_B)$ and we still get all
conclusions of the A)-part. Therefore, we can also prove (ii) of
the following corollary to Theorem 25 of \cite{K}.

\begin{corollary} \label{introcorNB} Let $B=R/I_B$ be a codimension two CM
  quotient of $R$, let $U= \Proj(B)-Z \hookrightarrow \PP^n$ be a l.c.i. and
  suppose $\depth_{I(Z)} B \geq 4$. Let $A$ be given by a regular section of
  $\widetilde{N_B}^*(s)$ on $U$, let $\eta(v):= \dim (I_B/I_B^2)_v$, and put
  $$\epsilon: = \eta(s)+ \sum_{j=1}^{\mu -1} \eta( n_{2,j})-\sum_{i=1}^{\mu}
  \eta( n_{1,i}). \ $$

   i) Let $j_0$ satisfy $n_{2,j_0} = \max n_{2,j}$. If $s > n_{2,j_0} +
   \max_{j \ne j_0} n_{2,j} - \min n_{1,i}$ and $char(k) \neq 2$, then $X$ is
   a $p_Y$-generic unobstructed arithmetically Gorenstein subscheme of $
   \PP^n$ of codimension 4 and $\ \dim_{(X)} \Hi^{p(x)}(\PP^{n})= \epsilon$.

   ii) If $\ _s\!\Ext_B^1(N_B,A)=0$, $char(k)=0$, $s> \max n_{2,j}/2 $ and $(X
   \subset Y)$ is general, then $X$ is unobstructed, $ \ \dim_{(X)}
   \Hi^{p(x)}(\PP^{n}) = \epsilon + \delta$ and the codimension of the stratum
   in $\Hi^{p(x)}(\PP^{n})$ of subschemes given by (\ref{introMs}) is
   $_{0}\!\ext^1_B(I_B/I_B^2,I_{A/B})$. Moreover if $s > \max n_{2,j}+\max
   n_{1,i}-\min n_{1,i}$ (resp. $s> \max n_{2,j}$), then $$\
   _{0}\!\ext^1_B(I_B/I_B^2,I_{A/B})=\ _{-s}\!\ext_B^{1}(I_B/I_B^2,N_B)=
   \delta \ , $$ (resp. $\ _{0}\!\ext^1_B(I_B/I_B^2,I_{A/B})=\
   _{-s}\!\ext_B^{1}(I_B/I_B^2,N_B)$).
 \end{corollary}

Here $I_{A/B}= \ker(B \rightarrow A)$ and ``$X$ is $p_Y$-generic'' if there is
an open subset of $ \Hi^{p(x)}(\PP^{n})$ containing $(X)$ whose members $X'$
are subschemes of some closed $Y'$ with Hilbert polynomial $p_Y$. The stratum
in $\Hi^{p(x)}(\PP^{n})$ of subschemes given by (\ref{introMs}) around $(X)$
is defined by functorially varying both $B$, $M$ and the regular section
around $(B \rightarrow A)$ (see \cite{K}; the definition before Theorem 25 for
details). Indeed it is proved in \cite{K}; Lemma 29 that pairs of closed
subschemes $(X' \subset Y')$ of $\PP^{n}$, $X'=\Proj(A')$ and $Y'=\Proj(B')$,
obtained as in \eqref{introMs} contain {\it an open subset} $U \ni (X \subset
Y)$ in the Hilbert flag scheme $D$, and taking such a $U$ small enough, we may
define the mentioned stratum to be $p(U)$ where $p:D \rightarrow
\Hi^{p(x)}(\PP^{n})$ is the projection morphism induced by $(X' \subset Y')
\rightarrow (X')$. Thus ``$X$ is $p_Y$-generic'' essentially means that the
codimension of the stratum of subschemes given by \eqref{introMs} around $(X)$
is zero.

Note also that ``$(X \subset Y)$ is general'' means that it is the general
member of an irreducible (non-embedded) component of the Hilbert flag scheme
$D$. Since we in the corollary suppose $\depth_{I(Z)} B \geq 4$ and hence
$\depth_{\mathfrak m}A \geq 2$, this is equivalent to saying that $(B
\rightarrow A)$ is the general member of an irreducible (non-embedded)
component of the ``Hilbert flag scheme'' parameterizing pairs of quotients of
$R$ with fixed Hilbert functions. Indeed we can replace the schemes $\Gr(R)$
of \cite{K} by $ \Hi ^{p(x)}(\PP^{n})$ because we work with algebras of
$\depth$ {\it at least} 2 at ${\mathfrak m}$ (\cite{K79}; Remark 3.7 or
\cite{elli}).

\begin{proof} By the text  before the exact sequence
(\ref{dimHilb}), to prove (i) it suffices
  to show that $\ _{-s}\!\Ext^1_B(I_B/I_B^2,N_B)=0$. To see it we observe that
  $$\Ext^1_B(I_B/I_B^2,N_B) \simeq \Ext^1_B(T_B,K_B(n+1))$$ where $T_B:=
  \Hom_B(I_B/I_B^2,I_B/I_B^2)$ by \cite{K}; Remark 42. We consider the exact
  sequence
  \eqref{NB2codim2} and we define $F:= \ker(\oplus I_B/I_B^2 (n_{2,j}) \rightarrow
  N_B)$. Since $N_B$ is a maximal CM $B$-module and $ I_B/I_B^2 $ has codepth
  1 (i.e. $\Ext^i_B(I_B/I_B^2,K_B)=0$ for $i \ge 2$) by \cite{AH} or
  \eqref{I2codim2}, we get $\Ext^2_B(F,K_B)=0$. It follows that
 $$\Ext^1_B(\oplus I_B/I_B^2 (n_{1,i}),K_B(n+1)) \rightarrow \
 \Ext^1_B(T_B,K_B(n+1))$$ is surjective. Since $$\Ext^1_B(I_B/I_B^2 ,K_B(n+1))
 \simeq\ \Ext^3_R(I_B/I_B^2,R) \simeq \ \Ext^2_R(I_B^2,R)\ ,$$ it suffices to
 show $ _{-s}\! \Ext^2_R(I_B^2(n_{1,i}),R) = 0$ for any $i$. Looking to
 \eqref{I2codim2} it is enough to see $\ _{-s}\!\Hom(\wedge^2 (\oplus
 R(-n_{2,j}))(n_{1,i}),R)=0$. Since, however, $n_{2,j} + n_{2,j'}-n_{1,i}-s <
 0$ for any $i,j,j'$, $j \ne j'$ by assumption, we easily get this vanishing
 for any $i$ and hence $\ _{-s}\!\Ext^1_B(I_B/I_B^2,N_B)=0$. Finally note that
 the dimension formula follows from \eqref{dimHilb} and \eqref{NB2codim2}
 since we get $(K_B)_{t-2s}=0$ and $\delta =0$ from the proof of (ii).

 (ii) By (\ref{Kcodim2}) we have $(K_B)_{t-2s}=0$ provided $2s> \max n_{2,j}$. By the
 discussion before Corollary~\ref{introcorNB} we must prove $\
 _{t}\!\Ext_B^2(S^2(I_{A/B}(s)),K_B)=0$. Using the proof of \cite{K}; Lemma 28
 there is an exact sequence $$ 0 \rightarrow \
 _{t}\!\Ext_B^2(S^2(I_{A/B}(s)),K_B) \rightarrow \
 _{t}\!\Ext_B^2(S^2(N_B),K_B) \rightarrow \ _{s}\!\Ext_B^2(N_B,B) $$ induced
 by \eqref{Ns} where we have $ _{t}\!\Ext_B^2(S^2(N_B),K_B) \simeq \
 _{0}\!\Ext_B^2(N_B,N_B) \simeq \ _{0}\!\Ext_B^2(N_B,I_{A/B}(s))$ by
 \eqref{globalext}, \eqref{Ns} and the fact that $N_B$ is a maximal CM
 $B$-module. Indeed $ _{t}\!\Ext_B^2(S^2(N_B),K_B) \simeq
 \Ext_{\sO_U}^2(\widetilde {S^2(N_B)}\arrowvert_U,\widetilde
 {K_B}\arrowvert_U(t)) \simeq \Ext_{\sO_U}^2(\widetilde
 {N_B}\arrowvert_U,\widetilde{N_B}^* \otimes \widetilde {K_B}\arrowvert_U(t))
 \simeq \ _{0}\!\Ext_B^2(N_B,N_B)$ by \eqref{globalext}. Since $\
 \Ext_B^1(N_B,B) = 0$ by \eqref{globalext} and \eqref{I2codim2}, it follows
 that
  $$\ _{t}\!\Ext_B^2(S^2(I_{A/B}(s)),K_B) \simeq \ _{s}\!\Ext_B^1(N_B,A)$$
 which vanishes by assumption.

 It remains to prove the final statement. If we apply $\Hom(-,K_B)$ to the
 exact sequence (\ref{res}) and we use the exact sequence (\ref{Kcodim2}), we
 get $\ _{-2s}\!\Ext_R^{i}(I_B,K_B(t))=0$ and hence $\
 _{-2s}\!\Ext_B^{i}(I_B/I_B^2,K_B(t))=0$ for $i=0, 1$ provided $s > \max
 n_{2,j}$. Similarly we use $\Hom(-,N_B)$ and the exact sequence
 (\ref{NBcodim2}) to show that $\ _{-s}\!\Hom(I_B,N_B)=0$ provided $s > \max
 n_{2,j}+\max n_{1,i}-\min n_{1,i}$. We conclude by applying
 $\Hom_B(I_B/I_B^2,-)$ to \eqref{Ns}.
\end{proof}

 \begin{remark} \label{remcorNB}  If $\depth_{I(Z)} B \geq 4$ and $char(k)
    \neq 2$, we showed in \cite{K}; Remark 42 that
  $$\ _0\!\Ext_B^2(N_B,N_B)\ \simeq \ _0\!\Hom_B(I_B/I_B^2,
  \HH^3_{I(Z)}(I_B/I_B^2))\ \simeq \ _0\!\Hom_B(I_B/I_B^2,
  \HH^4_{I(Z)}(I_B^2)) \ .
  $$
 Similarly one shows $\ \Ext_B^2(N_B,B)\ \simeq \
  \HH^4_{I(Z)}(I_B^2)$. Hence the group $\ _s\!\Ext_B^1(N_B,A)$ of
  Corollary~\ref{introcorNB} is isomorphic to the kernel of the natural map $
  _0\!\Hom_B(I_B/I_B^2, \HH^4_{I(Z)}(I_B^2)) \rightarrow\
  _s\HH^4_{I(Z)}(I_B^2) $ induced by the regular section $\sigma$. This sometimes
  allows us to verify $\ _s\!\Ext_B^1(N_B,A)=0$.
 \end{remark}

 \begin{remark} \label{correct} The first author takes the opportunity to
   point out a missing assumption in \cite{K} as well as in \cite{K03}. In
   these papers there are several theorems involving the {\it codimension of a
     stratum} in which the assumption ``$(B \rightarrow A)$ is general'' or
   ``$(B)$ general'' is missing. The main result \cite{K03}; Theorem 5 (and
   hence \cite{K}; Theorem 15) uses generic smoothness in its proof and refers
   to \cite{KMMNP}; Proposition 9.14 where the ``general'' assumption occurs,
   as it should. In the proof of \cite{K03}; Theorem 5 we need ``$(B
   \rightarrow A)$ general'' to compute the dimension of the stratum. It is
   easily seen from the proof that what we really need is that ``$(B
   \rightarrow A)$ is general'' in the sense that, for a given $(B \rightarrow
   A)$, $\ _{0}\!\hom_R(I_B,I_{A/B}) \ $ obtains its least possible value in
   the irreducible components of $\Gr(H_B,H_A)$ to which $(B \rightarrow A)$
   belongs.
  Thus in \cite{K03}; Theorem 5, Proposition 13, Theorem 16 (and hence
  \cite{K}; Theorem 23), for the codimension statement we should assume
  ``$(B)$ general'' or at least that $\ _{-s}\!\hom_R(I_B,K_{B}) \ $ obtains
  its least possible value in the irreducible component of $\Gr(H_B)$ to which
  $(B)$ belongs. If we apply our results in a setting where these hom-numbers
  vanish (this is what we almost always do), we don't need to assume ``$(B)$
  or $(B \rightarrow A)$ general''.
\end{remark}

So Remark~\ref{correct} gives the reason for including the
assumption ``$(X \subset Y)$ is general'' in
Corollary~\ref{introcorNB} (ii) even though this assumption does
not occur in the codimension statements of the A)-part of Theorems
1 and 25 of \cite{K}.


\section{Ideals generated by submaximal minors of square matrices}

 Let $X=\Proj(A)\subset \PP^n$  be a
codimension 4, determinantal scheme defined by the submaximal
minors of a $t \times t$ homogeneous matrix. The goal of this
section is to compute the dimension of $\Hi^{p(x)}(\PP^{n} )$ for
$n\ge 5$ at $(X)$  in terms of the corresponding degree matrix.
The proof requires a proposition (valid for $n\ge 3$) on how $A$
is determined by a locally regular section of $I_B/I_B^2(s)$ where
$B=R/I_B$ is a codimension 2 CM quotient. Let us first fix the
notation we will use throughout this section.

\vskip 4mm  Given   a homogeneous matrix $\cA$, i.e. a matrix
representing a degree 0 morphism $\phi $ of free graded
$R$-modules, we denote by $I(\cA)$ (or $I(\phi )$) the ideal of
$R$ generated by the maximal minors of $\cA$ and by $I_j(\cA)$ (or
$I_j(\phi )$) the ideal generated by the $j \times j$  minors of
$\cA$.

\begin{definition}
 A codimension $c$ subscheme $X\subset \PP^{n}$ is
called a \emph{determinantal} scheme if there exist integers $r$,
$p$ and $q$ such that $c=(p-r+1)(q-r+1)$ and $I(X)=I_r(\cA)$ for
some $p\times q$ homogeneous matrix $\cA$. $X\subset \PP^{n}$ is
called a \emph{standard  determinantal} scheme if $r=min(p,q)$.
The corresponding rings $R/I_r(\cA)$ are called determinantal
(resp. standard determinantal) rings.
 \end{definition}

\vskip 2mm
 In this section, we will mainly focus our attention on determinantal
 schemes defined by the submaximal minors
 of a square homogeneous matrix. Let $X\subset \PP^{n}$ be a codimension
 4,
 determinantal scheme
defined by the vanishing of the submaximal minors of a $t\times t$
homogeneous matrix $\cA=(f_{ji})_{i,j=1,\cdots ,t}$ where
$f_{ji}\in { k}[x_{0},\cdots ,x_{n}]$ are homogeneous polynomials
of degree $a_j-b_{i}$ with $b_1 \le b_2\le \cdots  \le b_t$ and
$a_1 \le a_2\le \cdots  \le a_{t}$. We assume without loss of
generality that $\cA$ is minimal; i.e., $f_{ji}=0$ for all $i,j$
with $b_{i}=a_{j}$. If we let $u_{ji}=a_j-b_i$ for all $j=1, \dots
, t$ and $i=1, \cdots , t$, the matrix $\cU=(u_{ji})_{i,j=1,\cdots
 ,t}$ is called the {\em degree matrix} associated to
$X$.

 We
denote by $W^{t-1}_{t,t}(\underline{b};\underline{a})\subset
\Hi^{p(x)}(\PP^{n})$
 the locus of
determinantal schemes $X\subset \PP^{n}$ of codimension $4$
defined by the submaximal minors of a homogeneous square matrix
$\cA=(f_{ji})_{i,j=1,\cdots ,t}$  as above. Notice that
$W^{t-1}_{t,t}(\underline{b};\underline{a})\ne \emptyset $ if and
only if $u_{i-1,i}=a_{i-1}-b_i> 0$ for $i=2,...,t$.

\vskip 2mm Let $\cN $ be the matrix obtained by deleting the last
row, let $I_B=I_{t-1}(\cN)$ be the ideal defined by the maximal
minors of $\cN$ and let $I_A=I_{t-1}(\cA )$ be the ideal generated
by the submaximal minors of $\cA $. Set $A=R/I_A=R/I(X)$ and
$B=R/I_B$.

\begin{remark}\label{hypothesis}
If the entries of $\cA $ and $\cN $ are sufficiently general
polynomials of degree $a_i-b_j$, $1\le i , j \le t$, and
$a_{i-1}-b_{i}>0$ for $2\le i \le t$, then $B$ is a graded
Cohen-Macaulay quotient of codimension 2 and $A$ is a graded
Gorenstein quotient of codimension 4.
\end{remark}

The goal of this section is to compute, in terms of $a_j$ and
$b_i$,  the dimension of the determinantal locus
$W^{t-1}_{t,t}(\underline{b};\underline{a})\subset \Hi^{p(x)}
(\PP^{n} )$, where $p(x)\in \QQ [x]$ is the Hilbert polynomial of
$X$. Note that the Hilbert polynomial of $X$ can be computed
explicitly using the minimal free $R$-resolution of $R/I(X)$ given
by Gulliksen and Neg\aa rd in \cite{GN}, see (\ref{gn}). We will
also analyse whether the closure of
$W^{t-1}_{t,t}(\underline{b};\underline{a})$ in $\Hi^{p(x)}
(\PP^{n} )$ is a generically smooth, irreducible component of $\Hi
^{p(x)} (\PP^{n}).$ To this end, we consider
$$F:=\oplus _{i=1}^tR(b_i)\stackrel {\phi}{ \longrightarrow}
G:=\oplus _{j=1}^tR(a_j)$$ the morphism induced by the above
matrix $\cA$ and
$$F\stackrel {\phi_t}{ \longrightarrow} G_t:=\oplus
_{j=1}^{t-1}R(a_j)$$ the morphism induced by the matrix $\cN$
obtained by deleting the last row of $\cA$. The determinant of
$\cA $ is a homogeneous polynomial of degree
$$s: =\degr(\deter(\cA))=\sum_{j=1}^ta_{j}-\sum_{i=1}^tb_i,$$
and the degrees of the maximal minors of $\cN $ are $s+b_{i}-a_t$,
i.e. $I_B$ has the following minimal free $R$-resolution
\begin{equation}\label{complex1}
0\longrightarrow G_t^*(a_t-s)\stackrel {^t \cN}{ \longrightarrow}
F^*(a_t-s)\stackrel {\beta}{ \longrightarrow} I_B\longrightarrow
0.
\end{equation}

\begin{proposition}\label{sectionnormal} Suppose $char(k)=0$.

(i) Let $A=R/I_{t-1}(\cA )$ be a determinantal ring of codimension
4 where $\cA$ is a $t\times t$ homogeneous matrix and let
$B=R/I_{t-1}(\cN)$ be the standard determinantal ring associated
to $\cN$ where $\cN$ is the matrix obtained by deleting the last
row of $\cA$.
 Moreover, let $Z\subset \Proj(B)$ be a closed subset such that
$\Proj(B)-Z\hookrightarrow \PP^n$ is a l.c.i. and suppose
$\depth_{I(Z)}B\ge 2$.
 Then, there is a regular section $\sigma $ of
$(\widetilde{I_B}/\widetilde{I_B^2}(s))_{|\Proj(B)-Z}$ where
$s=\deg(\det (\cA))$ whose zero locus precisely defines $A$ as a
quotient of B (i.e. $\sigma $ extends to a map $\sigma: B
\longrightarrow I_B/I_B^2(s)$ such that $A=B/(\im \sigma^*)$).

(ii) Conversely, let $B=R/_{t-1}(\cN)$ be a standard determinantal
ring of codimension 2, let $Z\subset \Proj(B)$ be a closed subset
such that $\Proj(B)-Z\hookrightarrow \PP^n$ is a l.c.i. and
 $\depth_{I(Z)}B\ge 2$ and let
$A'$ be defined by a regular section $\sigma $ of
$(\widetilde{I_B}/\widetilde{I_B^2}(s))_{|\Proj(B)-Z}$, i.e. given
by \begin{equation}\label{j07} 0\longrightarrow
K_B(n+1-2s)\longrightarrow N_B(-s) \stackrel {\sigma ^*}{
\longrightarrow}B \longrightarrow A' \longrightarrow
0.\end{equation} Then, there is a $t\times t$ homogeneous matrix
$\cA '$ obtained by adding a row to $\cN$ such that
$I_{A'}=I_{t-1}(\cA ')$.
\end{proposition}

\begin{proof} First of all, to define $\sigma $, we consider the commutative diagram

\[
\begin{array}{ccccccccc}
& 0 & & & & & \\ & \downarrow & & & & & \\ 0 & \longrightarrow
 G_t^*(a_t-s) & \stackrel {^t \cN}{ \longrightarrow}
& F^*(a_t-s) & \stackrel {\beta}{ \longrightarrow} & I_B &
\longrightarrow & 0
\\ &
\ \ \downarrow \alpha& & \parallel & & \downarrow & \\
&G^*(a_t-s) & \stackrel {\phi^*(a_t-s)}{ \longrightarrow} &
F^*(a_t-s) & \longrightarrow & (\coker\phi^*)(a_t-s) &
\longrightarrow & 0
\\ &
\downarrow & & & & & \\ & R(-s) & & & & & \\ & \downarrow & & & &
& \\ &0 & & & & &
\end{array}
\]
where $\alpha:G_t^*(a_t-s)\hookrightarrow G^*(a_t-s)$ is the
natural inclusion defined by
{\footnotesize $\alpha \left ( \begin{array}{c} f_{1}  \\ \vdots \\
f_{t-1}\end{array} \right )=\left ( \begin{array}{c} f_{1}  \\ \vdots \\
f_{t-1}\\ 0 \end{array} \right )$} and $\beta $ is given by
multiplication with the maximal minors of the matrix $\cN$. Thus
we get  the exact sequence
\begin{equation}\label{1*} R(-s)\stackrel {\cdot \deter(\phi)}{
\longrightarrow} I_B \longrightarrow
(\coker\phi^*)(a_t-s)\longrightarrow 0\end{equation}  and hence
\begin{equation}\label{2**} (\coker\phi^*)(a_t) \simeq I_B(s)/\deter(\phi).\end{equation}

\noindent If we tensor $R(-s)\stackrel {\cdot \deter(\phi)}{
\longrightarrow} I_B$ with $B(s)$, we get a section $\sigma $ of
$I_B/I_B^2(s)$. Before proving that the zero locus of $\sigma $
defines precisely $A$ as a quotient of $B$ via $\im(\sigma
^*)=I_{A/B}$, we claim that any regular section $\sigma '$ of
$I_B/I_B^2(s)$ defining $A'$ via $A'=B/\im(\sigma '{} ^{*})$ gives
rise to a homogeneous matrix $\cA '$ and a corresponding map $\phi
'$ such that (\ref{1*}) and (\ref{2**}) hold with $\phi '$ instead
of $\phi$. Indeed, given a section $\sigma '$ of $I_B/I_B^2(s)$,
there exists a map $\sigma ''$ fitting into a commutative diagram
$$
\xymatrix{ &  F^*(a_t)\otimes B \ar@{>>}[d]
\\B  \ar[r]_{\sigma '} \ar[ru]^-{\sigma ''} &     I_B/I_B^2(s)  }$$
and we denote by $\sigma _R\in \Hom_R(F,R(a_t))$ the map which
corresponds to $\sigma ''(1)$. Since $\Hom_R(F,R(a_t))=\Hom
(\oplus _{i=1}^tR(b_i), R(a_t))$, the morphism $\sigma _R$
determines a $1\times t$ row $\underline{g}=(g_1,\cdots ,g_t)$
where $g_i$ is a homogeneous form of degree $a_t-b_i$, $1\le i \le
t$ and we define $\cA'=\left ( \begin{array}{c} \cN  \\
\underline{g} \end{array} \right )$. Since the vertical map in the
above diagram is induced by $\beta $ described above, we may
assume that $\det (\phi ')=\sigma '(1)$ modulo $I_B^2(s)$ and we
get the claim.

It remains to show that $\im (\sigma ^*)=I_{A/B}$ where
$I_A=I_{t-1}(\cA)$ and that $\sigma $ is a regular section. Note
that this will also show that $\im(\sigma '^*)=I_{A'/B}$ where
$I_{A'}=I_{t-1}(\cA ')$, i.e. we get the converse. Moreover
looking at the exact sequence (\ref{j07}) and recalling that
$$N_B\simeq K_B(n+1)\otimes I_B/I_B^2,$$ we see that $\im (\sigma
^*)=\coker (\sigma(2s)\otimes id)$ where
$id:K_B(n+1)\longrightarrow K_B(n+1)$ is the identity map and
$\sigma $ is induced by $\det (\phi)$. Since we get
$$F(s-a_t)\longrightarrow G_t(s-a_t) \longrightarrow
K_B(n+1)\longrightarrow0$$ by dualizing the exact sequence
(\ref{complex1}), we see that the cokernel above is the same as
the twisted cokernel of the composition $$ \gamma: G_t(-a_t)
\longrightarrow K_B(n+1-s)\stackrel {\sigma(s)\otimes id}{
\longrightarrow} N_B.$$ Hence, we must prove that $\coker
(\gamma)=I_{A/B}(s)$ where $I_A=I_{t-1}(\cA)$.

 By \cite{GN}; Th\'{e}or\`{e}me
2 (see also \cite{Ile}; Theorem 2), we have an exact sequence:

\begin{equation} \label{gn} \ker[\Hom(F,F)\oplus \Hom(G,G) \stackrel {j}{
\longrightarrow}R] \longrightarrow \Hom(F,G) \longrightarrow I_A
(s) \longrightarrow 0\end{equation}

\noindent where $j(\rho_0,\rho_1)=tr(\rho _0)-tr(\rho_1)$ and $tr$
is the trace map. The map $\Hom(F,G)\longrightarrow I_A(s)$ is
given by $\gamma\longrightarrow tr(\gamma\psi)$ where $\psi $ is
the matrix of cofactors, i.e. this map is given by the submaximal
minors of $\cA$ while the map $\Hom(F,F)\oplus \Hom(G,G)\stackrel
{\eta}{ \longrightarrow} \Hom(F,G)$ is given as a difference of
the obvious compositions with $\phi $, i.e., $\eta (\rho _0,\rho
_1)=\rho _1\phi-\phi\rho _0$. Since we have
\[
\begin{array}{ccccccccc}
 \Hom(F,F)\oplus \Hom(G,G) & \stackrel {\eta}{
\longrightarrow}
\Hom(F,G)\longrightarrow & I_A(s) \\
(id,0)  & \mapsto & t\cdot \deter(\phi ) \end{array}\] and since
there is a commutative diagram
\[
\begin{array}{ccccccccc}0 & & & & & \\ \downarrow & & & & & \\
\ker(j) &  \longrightarrow & \Hom(F,G) &  \longrightarrow & I_A(s)
& \longrightarrow & 0
\\
\downarrow & & \parallel & & \downarrow & \\
\Hom(F,F)\oplus \Hom(G,G)  & \stackrel {\eta}{ \longrightarrow} &
\Hom(F,G) &  \longrightarrow & \coker\eta & \longrightarrow & 0
\\
\downarrow & & & & & \\
R & & & & &
\end{array}
\] we get an exact sequence
$$R\stackrel {\cdot t\deter(\phi)}{ \longrightarrow} I_A(s)
 \longrightarrow  \coker\eta  \longrightarrow  0.$$ Hence,
 $\coker(\eta)\simeq I_A(s)/\deter(\phi)$ ($char(k)=0$) and the
 following sequence is exact:
$$\Hom(F,F)\oplus \Hom(G,G)   \stackrel {\eta}{ \longrightarrow}
\Hom(F,G)   \longrightarrow  I_A(s)/\deter(\phi) \longrightarrow
0.$$

Now we look at the commutative diagram

$$
\xymatrix{ \Hom(R(-a_t),G^*) \ar[d]_{\simeq} \ar[r] &
\Hom(R(-a_t),F^*) \ar[d]_{\simeq} \ar[r] & I_B(s)/
\deter(\phi)\ar[dd] \ar[r] & 0
\\ \Hom(G,R(a_t))\ar[d]^{(0, \cdot )} \ar[r] & \Hom(F,R(a_t)) \ar[d] & &
\\
\Hom(F,F)\oplus \Hom(G,G) \ar[d]^{(id, \alpha_1^*)} \ar[r]^-{\eta}
& \Hom(F,G) \ar[d]^{\alpha_2^*} \ar[r] & I_A(s)/
\deter(\phi)\ar[d] \ar[r] & 0
\\
\Hom(F,F)\oplus \Hom(G,G_t) \ar[r]^-{\eta_t} & \Hom(F,G_t) \ar[r]
& \coker(\eta_t) \ar[r]  & 0 }$$

\noindent where $\alpha _1^*$ and $\alpha_2^*$ are induced by
$\alpha$ in a natural way and $\eta _t$ is a difference of the
obvious compositions, i.e., $\eta
_t(\rho_0,\rho_1')=\rho_1'\phi-\phi _t\rho_0$. We see, in
particular, that the ideal $I_{A/B}=I_A/I_B$ is given by an exact
sequence

$$\Hom(F,F)\oplus \Hom(G,G_t)   \stackrel {\eta _t}{ \longrightarrow}
\Hom(F,G_t)   \longrightarrow  I_{A/B}(s)  \longrightarrow 0$$
where  the rightmost map is given by the submaximal minors of the
matrix $\cA$ which do not belong to $I_B$.

On the other hand, by (\ref{NBcodim2}), there is an exact sequence
$$\Hom(G_t^*,G_t^*)  \oplus \Hom(F^*,F^*) \longrightarrow
\Hom(G_t^*,F^*)   \longrightarrow N_B \longrightarrow 0$$ or,
equivalently,
$$ \Hom(F,F)\oplus \Hom(G_t,G_t)    \stackrel {\eta'}{ \longrightarrow}
\Hom(F,G_t)    \longrightarrow  N_B \longrightarrow 0$$ where
$N_B=\Hom(I_B/I_B^2,B)$ is the normal module and $\eta '$ is given
by $\eta'(\rho_0,\rho_2)=\rho_2\phi_t-\phi_t\rho_0$. Using again
the exact sequence $0\longrightarrow R(a_t)\longrightarrow
G\stackrel {\alpha^*}{ \longrightarrow} G_t\longrightarrow 0$ we
get a commutative diagram

$$
\xymatrix{ \Hom(F,F)\oplus \Hom(G_t,G_t) \ar[d]^{(id, \alpha_3)}
\ar[r]^-{\eta'} & \Hom(F,G_t) \ar@{=}[d] \ar[r] & N_B \ar[d]
\ar[r] & 0
\\
\Hom(F,F)\oplus \Hom(G,G_t) \ar[d] \ar[r]^-{\eta_t} & \Hom(F,G_t)
\ar[r] & I_{A/B}(s) \ar[r]  & 0 \\ \Hom(R(a_t),G_t) }$$

\noindent where $\alpha_3$ is induced by $\alpha$. Hence we get an
exact sequence
$$ \Hom(R(a_t),G_t) \stackrel {\gamma }{ \longrightarrow} N_B\longrightarrow
I_{A/B}(s)\longrightarrow 0.$$

This proves that $\coker(\gamma )=I_{A/B}(s)$, i.e.
$\im(\sigma^*)=I_{A/B}$ as required.  Finally note that the above
codimension and depth relations imply that $\sigma $ is a {\em
regular} section on $U:=\Proj(B)-Z$ because $(\im
\tilde{\sigma}^*)_{|U}$ must locally on $U$ be generated by two
regular elements (to get that $(B/\im \tilde{\sigma }^*)_{|U}$ is
a codimension 2 Cohen-Macaulay quotient of $\widetilde{B}{|U}$).
This completes the proof of Proposition \ref{sectionnormal}.
\end{proof}

\vskip 2mm The above proposition seems to be known in special
cases. We have for instance noticed that Ellingsrud and Peskine
claim that the Artinian Gorenstein ring associated to an
invertible sheaf $\cO _S(C)$ on a surface $S$ in $\PP^3$, where
$C$ is an arithmetically CM curve, is given by the submaximal
minors of a square matrix which extends the Hilbert-Burch matrix
associated to $C$ in $\PP^3$ (see the text of \cite{EP} before
Proposition 6). Since, we get (\ref{introMs}) with $M=N_B$ by
applying $H^0_*(-)$ to the exact sequence
$$ 0 \rightarrow \cN _{C/S}(-s) \rightarrow \cN _C(-s) \rightarrow
\cN_S|_C(-s)\simeq \cO_C\rightarrow 0$$ of normal sheaves, it is
clear that their Gorenstein ring (see their construction 2) is
essentially the same as ours in the Artinian case.  However, we
have given a complete proof of the above proposition suited to our
applications.

 \vskip 2mm As a nice application of   Proposition \ref{sectionnormal}
 we have

 \vskip 2mm \begin{proposition} \label{glicci}
 Let $X\subset \PP^n$, $n\ge 4$,  be a codimension 4 scheme defined by the submaximal
 minors of a $t\times t$ homogeneous square matrix $\cA$. Then
 $X$ is in the Gorenstein liaison class of a complete intersection, i.e. $X$ is glicci.
\end{proposition}
\begin{proof} By \cite{GN}; Th\'{e}or\`{e}me 2 (see also
Proposition \ref{sectionnormal}), $X$ is arithmetically Gorenstein
and hence glicci (\cite{cdh}; Theorem 7.1).
\end{proof}

\begin{remark} The above proposition has been recently generalized
by Gorla. In \cite{Go}; Theorem 3.1, she has proved that any
codimension $(t-r+1)^2$ ACM scheme $X\subset \PP^n$ defined by the
$r\times r$ minors of a $t\times t$ homogeneous square matrix $\cA
$ is glicci.

For an introduction to glicciness, see \cite{KMMNP}.
\end{remark}

We are now ready to compute $\dim
W^{t-1}_{t,t}(\underline{b};\underline{a})$ and $\dim_{(X)}\Hi
^{p(x)}\PP^n$, $n\ge 5$,  in terms of $a_1, \cdots ,a_t$ and $b_1,
\cdots , b_{t}$. Note that if $t=2$ then a general  $X$ is a
complete intersection in which case these dimensions are well
known.

\begin{theorem}\label{mainthm1} ($char(k)=0$) Fix integers $a_1\le a_2 \le \cdots \le a_t$ and
$b_1\le b_2 \le  \cdots \le b_{t}$. Assume $t>2$, $ a_i\ge
b_{i+3}$ for $1\le i \le t-3$ (and $a_1\ge b_t$ if $t= 3$),
$a_t>a_{t-1}+a_{t-2}-b_1$ and $n\ge 5$. Then
$W^{t-1}_{t,t}(\underline{b};\underline{a})$ is irreducible.
Moreover, if  $(X)$ is general in
$W^{t-1}_{t,t}(\underline{b};\underline{a})$, then $X$ is
unobstructed,
 and
$$\dim  W^{t-1}_{t,t}(\underline{b};\underline{a})
= \dim _{(X)}\Hi ^{p(x)}(\PP^n)=$$ {\footnotesize $$\sum _{1\le i
,j\le t }{a_j-b_i+n\choose n}- \sum _{1\le i\le t-1 \atop 1\le j
\le t}{a_j-a_i+n\choose n}-\sum _{1\le i , j \le
t}{b_i-b_j+n\choose n}+ \sum _{1\le i\le t \atop 1\le j \le
t-1}{b_i-a_j+n\choose n}$$
$$-\sum _{1\le j \le t \atop
1\le i \le k \le t}{a_t-s-b_i-b_k+a_j+n\choose n}+\sum _{1\le i, j
\le t \atop 1\le k \le t-1}{a_t-s-b_i-a_k+a_j+n\choose n}-$$$$
\sum _{1\le i < k \le t-1\atop 1\le j \le
t}{a_t-s-a_i-a_k+a_j+n\choose n}+ \sum _{2\le i\le
t}{a_t-s+b_i-2b_1+n\choose n}.$$ }

\end{theorem}
\begin{proof}
Let  $X\subset \PP^{n}$ be an arithmetically Gorenstein scheme of
codimension $4$ defined by the submaximal minors of a homogeneous
square matrix $\cA=(f_{ji})_{j=1,...,t}^{i=1,...,t}$ where
$f_{ji}\in { k}[x_{0},...,x_{n}]$ is a sufficiently general
homogeneous polynomial of degree $a_j-b_{i}$ and let $Y\subset
\PP^n$ be a codimension 2 subscheme defined by the maximal minors
of the matrix $\cN $ obtained deleting the last row of $\cA $ (see
Remark \ref{hypothesis}). So, the homogeneous ideal $I_B=I(Y)$ of
$Y$ has the following minimal free $R$-resolution

\begin{equation}\label{burch1}
0\longrightarrow F_2=\oplus _{j=1}^{t-1}R(a_t-s-a_j)\stackrel {^t
\cN}{ \longrightarrow} F_1=\oplus _{i=1}^tR(a_t-s-b_i
)\longrightarrow I_B\longrightarrow 0.       \end{equation}

By Proposition \ref{sectionnormal}, $X$ is the zero locus of a
suitable regular section of $\widetilde{I_B}/\widetilde{I_B^2}(s)$
where $s=\degr(\deter(\cA ))$ and
$W^{t-1}_{t,t}(\underline{b};\underline{a})$ is irreducible  by
\cite{K}; Corollary 41. Since the hypothesis
$a_t>a_{t-1}+a_{t-2}-b_1$ is equivalent to
 $$s > s+a_{j_0}-a_t +
   \max_{1\le j \le t-1 \atop j\ne j_0}(s+a_{j}-a_{t})  - \min _{1\le i \le t}(s+b_{i}-a_{t})
   ,$$  where $s+a_{j_0}-a_t = \max _{1\le j\le t-1}
   (s+a_{j}-a_{t})$;
and since $a_i\ge b_{i+3}$ for $1\le i\le t-3$ (and $a_1\ge b_t$
if $t=3$) implies that $B:=R/I_B$ given by (\ref{burch1})
satisfies $\depth_{I(Z)}B\ge 4$ (cf. \cite{KM}; Remark 2.7), we
can apply Corollary \ref{introcorNB} and we get that $X$ is
unobstructed and
$$\dim W^{t-1}_{t,t}(\underline{b};\underline{a})= \dim _{(X)}\Hi
^{p(x)}(\PP^n)=\eta(s)+\sum_{j=1}^{t-1}\eta
(n_{2,j})-\sum_{i=1}^t\eta (n_{1,i})$$ where $\eta(t)=\dim
(I(Y)/I(Y)^2)_t=\dim I(Y)_t-\dim I(Y)^2_t$, $n_{2,j}=s+a_j-a_t$,
$1\le j\le t-1$, and $n_{1,i}=s+b_i-a_t$, $1\le i \le t$. By
(\ref{I2codim2}), $I(Y)^2$ has a minimal free $R$-resolution of
the following type:

\begin{equation}\label{S2I}
0\longrightarrow \wedge ^2F_2=\oplus _{1\le i<j\le
t-1}R(-a_i-a_j+2a_t-2s) \longrightarrow
\end{equation} $$F_1\otimes F_2=\oplus _{1\le i \le t \atop 1\le j \le
t-1}R(-b_i-a_j+2a_t-2s) \longrightarrow $$$$ S^2F_1=\oplus _{1\le
i\le j\le t}R(-b_i-b_j+2a_t-2s)\longrightarrow
I(Y)^2\longrightarrow 0.
$$

Using (\ref{burch1}) and (\ref{S2I}), we obtain

$$
\eta(s)=\sum _{1\le i \le t}{a_t-b_i+n\choose n}- \sum _{1\le i\le
t-1}{a_t-a_i+n\choose n}-\sum _{1\le i \le j \le
t}{2a_t-s-b_i-b_j+n\choose n}+ $$ $$\sum _{1\le i\le t \atop 1\le
j \le t-1}{2a_t-s-b_i-a_j+n\choose n}-\sum _{1\le i < j \le
t-1}{2a_t-s-a_i-a_j+n\choose n}.$$

Using again (\ref{burch1}) and (\ref{S2I}), we get

$$
\sum_{j=1}^{t-1}\eta (n_{2,j})-\sum_{i=1}^t\eta (n_{1,i})=\sum
_{1\le i \le t \atop 1\le j \le t-1}{a_j-b_i+n\choose n}-\sum
_{1\le i\le t-1 \atop 1\le j \le t-1}{a_j-a_i+n\choose n}-$$
$$\sum
_{1\le j \le t-1 \atop 1\le i \le k \le
t}{a_t-s-b_i-b_k+a_j+n\choose n}+ \sum _{1\le i\le t \atop 1\le j,
k \le t-1}{a_t-s-b_i-a_k+a_j+n\choose n}-$$
$$\sum _{1\le i < k \le t-1\atop 1\le j \le
t-1}{a_t-s-a_i-a_k+a_j+n\choose n} -\sum _{1\le i \le t \atop 1\le
j \le t}{b_i-b_j+n\choose n}+\sum _{1\le i\le t \atop 1\le j \le
t-1}{b_i-a_j+n\choose n}+$$$$\sum _{1\le i\le t \atop 1\le j \le k
\le t}{a_t-s+b_i-b_j-b_k+n\choose n}- \sum _{1\le i, k\le t \atop
1\le j \le t-1 }{a_t-s+b_i-b_k-a_j+n\choose n}+$$
$$\sum _{1\le k < j \le t-1\atop 1\le i \le
t}{a_t-s+b_i-a_k-a_j+n\choose n}.$$

\vskip 2mm Since $a_{i-1}>b_i$ and $a_{i}\ge b_{i+3}$ for $1\le i
\le t-3$ (and $a_1\ge b_t$ if $t= 3$), by hypothesis, the last two
sums of binomials vanish. Indeed, to see that
$a_t-s+b_i-b_k-a_j<0$ (resp. $a_t-s+b_i-a_k-a_j<0$) for $1\le
i,k\le t$ and $1\le j\le t-1$ (resp. $1\le i\le t$ and $1\le
k<j\le t-1$), it suffices to show that
$b_t-b_1-a_1<s-a_t=a_1+a_2+\cdots +a_{t-1}-b_1-b_2-\cdots -b_t$
(resp. $b_t-a_1-a_1<s-a_t=a_1+a_2+\cdots +a_{t-1}-b_1-b_2-\cdots
-b_t$) which is rather straightforward to prove.

Moreover, the same type of argument applies to see that
$a_t-s+b_i-b_j-b_k<0$ for all $1\le i \le t$ and $1\le j<k\le t$
and we can replace the summand $\sum _{1\le i\le t \atop 1\le j
\le k \le t}{a_t-s+b_i-b_j-b_k+n\choose n}$ by $\sum _{2\le i\le t
}{a_t-s+b_i-2b_1+n\choose n}$. Putting all together we get
$$\dim W^{t-1}_{t,t}(\underline{b};\underline{a}) = \dim _{(X)}\Hi
^{p(x)}(\PP^n) =$$ $$\sum _{1\le i \le t \atop 1\le j \le
t}{a_j-b_i+n\choose n}-\sum _{1\le i\le t-1 \atop 1\le j \le
t}{a_j-a_i+n\choose n}-\sum _{1\le i \le t \atop 1\le j \le
t}{b_i-b_j+n\choose n}+$$
$$\sum _{1\le i\le t \atop 1\le j \le t-1}{b_i-a_j+n\choose n}-
\sum _{1\le j \le t \atop 1\le i \le k \le
t}{a_t-s-b_i-b_k+a_j+n\choose n}+$$
$$ \sum _{1\le i, j \le t \atop 1\le k \le t-1}{a_t-s-b_i-a_k+a_j+n\choose n}- \sum
_{1\le i < k \le t-1\atop 1\le j \le t}{a_t-s-a_i-a_k+a_j+n\choose
n}+$$ $$\sum _{2\le i\le t}{a_t-s+b_i-2b_1+n\choose n}.$$
\end{proof}

\section{Examples}

We will end this work with some examples where we use Theorem
\ref{mainthm1}. Moreover, these examples show that  the hypothesis
$a_t>a_{t-1}+a_{t-2}-b_1$ cannot be avoided! To handle such cases,
we prove a proposition which estimates the codimension of the
stratum in $\Hi ^{p(x)}(\PP^n)$ of subschemes given by the exact
sequence  (\ref{introMs}).

\begin{example} Let $R=k[x_0,\cdots ,x_5]$ and let
$X=\Proj(A)\subset \PP^5=\Proj(R)$ be a general arithmetically
Gorenstein curve defined by the submaximal minors of a $4\times 4$
matrix whose first 3 rows are linear forms and whose last row are
forms of degree $s-3$ ($s\ge 4$), i.e. $b_i=0$ for $1\le i\le 4$,
$a_j=1$ for $1\le j\le 3$ and $a_4=s-3$. Then, Theorem
\ref{mainthm1} applies provided $s>5$ and we get that $X$ is
unobstructed and
$$\dim W_{4,4 }^3(\underline{0};1,...,1,s-3)=\dim_{(X)}\Hi ^{p(x)}(\PP^5)=12 {6\choose 5}+4{s+2\choose 5}-9{5\choose
  5}-3{s+1\choose 5}-$$ $$16{5\choose 5}-10{s-1\choose 5}+12{s-2\choose
  5}-3{s-2\choose 5}= 2s^3-10s^2+13s+48.$$

Moreover deleting the last row and taking maximal minors, we get a
threefold $Y=\Proj(B)$ with resolution
\begin{equation} \label{resY} 0\longrightarrow  R(-4)^3 \longrightarrow R(-3)^4
\longrightarrow R \longrightarrow B \longrightarrow 0,
\end{equation}
leading to $$H_B(\nu )={\nu +3\choose 3}+2{\nu +2\choose 3}+3{\nu
+1\choose 3}=p_Y(\nu) \mbox{ for }\nu \ge 0.$$  Since $A$ is given
by (\ref{introMs}) with $t=6$ and $M=N_B$, we get $\cO_X\simeq
\omega_X(2s-6)$. Hence $h^1(\cO_X(s-3))=h^0(\cO_X(s-3))$ and the
Hilbert polynomial of $X$ must be of the form $p_X(\nu)=d\nu
+1-g=d(\nu +s-3)$. Looking to (\ref{resY}) we get
$$p_X(s-2)=h^0(\cO_X(s-2))-h^0(\cO_X(s-4))=
$$ $$=h^0(\cO_Y(s-2))-h^0(\cO_Y(s-4))=6s^2-28s+36,$$ i.e.
$d=deg(X)=6s^2-28s+36$ and $g=1+d(s-3)$.

Note that Theorem \ref{mainthm1} takes care of all cases except for $s=4$ and
$s=5$. For these two values of $s$, we can, however, use
Corollary~\ref{introcorNB} (ii) to find $\dim_{(X)}\Hi ^{p(x)}(\PP^5)$ because
$$_0\! \Ext^2_B(N_B,N_B) \simeq \ _0\! \Hom(I_B/I_B^2,H^4_{\goth m}(I_B^2))=0$$
by (\ref{resY}) and Remark \ref{remcorNB}. Indeed $_3H^4_{\goth
  m}(I_B^2)\hookrightarrow \ _3H^6_{\goth m}(R(-8)^3)=0$ by (\ref{I2codim2}).
Hence $X$ is unobstructed,
$$\dim_{(X)}\Hi ^{p(x)}(\PP^5)=2s^3-10s^2+13s+48+\delta$$ where
$\delta=\delta(K_B)_{6-2s}-\delta(N_B)_{-s}$, and moreover, if
$s=5$, then $\delta $ is the codimension of the closure of $W_{4,4
}^3$ in $\Hi^{p(x)}(\PP^5)$. We claim that
$$(\delta(K_B)_{6-2s},\delta(N_B)_{-s})=\begin{cases} (-3,-15) & \mbox{for
}s=4 \\ (0,-12) & \mbox{for }s=5, \end{cases} $$ i.e., $\delta=12$
in both cases.

To find $\delta(K_B)_{6-2s}$ we apply $\Hom_B(-,K_B(6))$ to
(\ref{alghom}) and we get $\ _{-2s}\!\Hom_B(I_B/I_B^2,K_B(6))=0$
and $\ _{-2s}\!\Ext_B^1(I_B/I_B^2,K_B(6))=\
_{-2s}\!\Hom(H_1,K_B(6)).$ Since the rank of $H_1$ is 2, we have
\begin{equation}\label{aux11}
\Hom(H_1,K_B(6))\simeq H_1(\sum_in_{1,i})=H_1(12)
\end{equation}
by \cite{AH} or \cite{KP}; Theorem 8, see the isomorphism
accompanying (\ref{introMs}). Using (\ref{K*codim2}) or more
precisely the exactness of
\begin{equation}\label{aux12} \wedge ^2(R(-3)^4)\longrightarrow
R(-4)^3 \longrightarrow H_1\longrightarrow 0
\end{equation}
(cf. \cite{AH}), we get
$$\delta(K_B)_{6-2s}=-\dim H_1(12)_{-2s}=\begin{cases} -3 &
\mbox{for }s=4 \\ 0 & \mbox{for }s=5. \end{cases} $$

It remains to compute $\delta(N_B)_{-s}$. If we dualize the exact
sequence (\ref{alghom}) we get
$$ 0 \longrightarrow N_B\longrightarrow B(3)^4\longrightarrow
H_1^*\longrightarrow
0$$ to which we apply $\ _{-s}\!\Hom(I_B/I_B^2,-)$. Combining with
$$ \ _{-s}\!\Hom(I_B/I_B^2,H_1^*)\simeq \
  _{-s}\!\Hom(I_B/I_B^2\otimes K_B(6),H_1^*\otimes K_B(6)) \simeq \
  _{-s}\!\Hom(N_B,H_1(12)) $$
 where again we have used
(\ref{aux11}), we get
$$\delta(N_B)_{-s}=4\dim(N_B)_{3-s}-\dim(_{-s}\Hom(N_B(-12),H_1)).$$

\noindent Using (\ref{K*codim2}), we see that  {\small
\begin{equation*}0 \rightarrow \
_{-s}\!\Hom(N_B(-12),K_B(6)^*)\rightarrow \ _{-s}\!
\Hom(N_B(-12),B(-4)^3)\rightarrow \
_{-s}\!\Hom(N_B(-12),H_1)\rightarrow 0
\end{equation*}} is exact because we have
$\Ext^1_B(I_B/I_B^2\otimes K_B,K_B^*)=0$ by \cite{KM}; Lemma 4.9.
Using \cite{KM}; (4.17) we also get the surjectivity of the
natural map $K_B^*\otimes B(-4)^4 \longrightarrow
\Hom_B(I_B/I_B^2\otimes K_B,K_B^*)$. Since we may use
(\ref{aux12}) to see that $(H_1)_\nu\simeq R(-4)^3_\nu\simeq
B(-4)^3_\nu$ for $\nu\le 5$, we get $K_B(6)^*_\nu =0$ for $\nu \le
5$ by (\ref{K*codim2}) and hence
$$\ _{-s}\!\Hom(N_B(-12),K_B^*(-6))\simeq \ _{-s}\!\Hom(I_B/I_B^2\otimes
K_B(6),K_B^*(6))=0$$ for $s\ge 4$. It follows that
$$\ _{-s}\!\Hom(N_B(-12),H_1)\simeq (I_B/I_B^2)^3_{8-s}$$ for $s\ge 4$
which implies (cf. (\ref{NBcodim2}) and (\ref{I2codim2})) that
$$\delta(N_B)_{-s} =\begin{cases} 4\dim(N_B)_{-2}-3\dim(I_B)_3=-12
&
  \text{ for } s=5 \\
  4\dim(N_B)_{-1}-3\dim(I_B)_4=-15 & \text{ for } s=4.
\end{cases}$$ Putting all together we get
$$ \dim_{(X)}\Hi^{p(x)}(\PP^5)=\begin{cases}
  2s^3-10s^2+13s+48=\dim W_{4,4 }^3(\underline{0};\underline{1}) & \text{ for
  } s>5
  \\ 125 & \text{ for } s=5 \\
  80 & \text{ for } s=4. \end{cases}$$ Moreover, applying Corollary
\ref{introcorNB} (ii), we get $\codim
_{\Hi^{p(x)}(\PP^5)}W_{4,4}^3(\underline{0};1,1,1,2)=12$ in the
case $s=5$. Finally, for $s=4$, using Macaulay 2 program
\cite{Mac} we have computed the dimension  $\
_0\!\hom(I_B,I_{A/B})=3$ for $(B \rightarrow A)$ general and hence
$\codim
_{\Hi^{p(x)}(\PP^5)}W_{4,4}^3(\underline{0};\underline{1})=\
_0\!\hom(I_B,I_{A/B}) + \delta=15$.
\end{example}

If $a_t \le a_{t-1}+ a_{t-2}-b_1$ we see in the example above that
$ W_{t,t}^{t-1}(\underline{b};\underline{a})$ is a proper closed
irreducible subset, i.e. the generic curve of the component of
$\Hi^{p(x)}(\PP^5)$ to which $
W_{t,t}^{t-1}(\underline{b};\underline{a})$ belongs is not defined
by submaximal minors of a matrix of forms of degree $a_j-b_i$. The
converse inequality always implies $\dim
W_{t,t}^{t-1}(\underline{b};\underline{a}) = \dim_{(X)}
\Hi^{p(x)}(\PP^n)$ by Theorem \ref{mainthm1}. The pattern above
for small $a_t$ may be typical, but is in general rather difficult
to prove. We illustrate this by two more examples.

\begin{example} Let $X=\Proj(A)\subset \PP^5$ be a general arithmetically
  Gorenstein curve defined by the submaximal minors of a $3\times 3$ matrix
  whose first 2 rows are linear forms and whose last row are forms of degree
  $s-2$ ($s\ge 3$), i.e. $b_i=0$ for $1\le i\le 3$, $a_j=1$ for $1\le j\le 2$
  and $a_3=s-2$. Thanks to
  Proposition~\ref{sectionnormal} 
  the analysis of \cite{K}; Example 43, immediately transfers to our case.
  Hence, for $s>4$ (i.e., $a_t > a_{t-1}+ a_{t-2}-b_1$), we see that $X$ is
  unobstructed and
$$\dim W_{3,3 }^4(\underline{0};1,...,1,s-2)=\dim_{(X)}\Hi^{p(x)}(\PP^5)=
(s+1)(s-1)^2+23.$$ Since  by deleting the last row and taking
maximal minors we get a threefold $Y=\Proj(B)$ for which $\ _0\!
\Ext^2_B(N_B,N_B)=0$, we have the unobstructedness of $X$ also for
$s=3,4$, and
$$(\delta(K_B)_{6-2s},\delta(N_B)_{-s})=\begin{cases} (-1,2) & \mbox{for
  }s=3 \\ (0,-3) & \mbox{for }s=4 \end{cases} .$$ That is, $\delta=-3$ when $s=3$, and $\delta=3$
  when $s=4$. In both cases,
$$\dim_{(X)}\Hi^{p(x)}(\PP^5)=(s+1)(s-1)^2+23 + \delta .$$ Thus
$$ \dim_{(X)}\Hi ^{p(x)}(\PP^5)=\begin{cases}
  36 & \text{ for } s=3 \\
  71 & \text{ for } s=4, \end{cases}$$ see \cite{K}; Example 43 for the
computations. Now, applying Corollary \ref{introcorNB} (ii), we
get $\codim _{\Hi ^{p(x)}(\PP^5)}$
$W_{3,3}^2(\underline{0};1,1,2)=3$ in the case $s=4$. Finally, for
$s=3$, a Macaulay 2 computation shows $\ _0\!\hom(I_B,I_{A/B})=3$
and hence $$\codim
_{\Hi^{p(x)}(\PP^5)}W_{3,3}^2(\underline{0};\underline{1})= \
_0\!\hom(I_B,I_{A/B}) + \delta=0 \ ! $$
\end{example}

In the above examples  we were able to analyse the case $a_t \le a_{t-1}+
a_{t-2}-b_1$ through Corollary~\ref{introcorNB} (ii) because $\
_{s}\!\Ext_B^1(N_B,A)=0$. Since this vanishing may be rare, we want to improve
upon Corollary~\ref{introcorNB} (ii), at least to get estimates of the
codimension of the stratum. We prefer to do it in the generality of Theorem 25
of \cite{K} to extend Theorem 25 in this direction. This leads to the
proposition below. Indeed with assumptions as in Proposition \ref{introcodimprop}, one knows
that the projection morphism $q: D \rightarrow \Hi^{p_Y}(\PP^{n})$ induced by
$(X' \subset Y') \rightarrow (Y')$ is smooth at $(X \subset Y)$ (\cite{K};
Theorem 47). Using that the corresponding tangent map is surjective, we get
Proposition~\ref{introcodimprop} and Remark~\ref{introproprem} (a). Since we
only use these results in Example~\ref{example3} and Remark~\ref{remcor2NB},
we skip the details of the proof which are rather straightforward once having
the results and proofs of \cite{K}. Put
$$ c(I_{A/B}):=\ _{0}\!\ext^1_B(I_B/I_B^2,I_{A/B})- \
_{t}\!\ext_B^2(S^2(I_{A/B}(s)),K_B).$$

\begin{proposition} \label{introcodimprop} Let $B = R/I_B$ be a graded licci
  quotient of $R$, let $M$ be a graded maximal Cohen-Macaulay $B$-module, and
  suppose $\widetilde{M}$ is locally free of rank $2$ in $U:=\Proj(B) -Z$,
  that $\dim { B}-\dim {B}/{I}(Z)\geq 2$ and $\wedge^2{\widetilde {
      M}}\arrowvert_U \simeq \widetilde {K_B}(t)\arrowvert_U$. Let $A$ be
  defined by a regular section $\sigma$ of $\widetilde{M}^*(s)$ on $U$, i.e.
  given by \eqref{introMs}, let $X=\Proj(A)$ and suppose $\
  _s\!\Ext^1_B(M,B)=0$ and $\dim { B} \geq 4$. Moreover let $char(k)=0$, let
  $(B \rightarrow A)$ be general and suppose $(M,B)$ is unobstructed along any
  graded deformation of $B$ and $\ _{-s}\!\Ext^2_B(I_B/I_B^2,M)= 0 \ $. Then
  the codimension, $codi$, of the stratum in $\Hi^{p(x)}(\PP^{n})$ of
  subschemes given by (\ref{introMs}) around $(X)$ satisfies
$$\ c(I_{A/B}) \le codi \le c(I_{A/B})+\ _{0}h^{2}(R,A,A) \le \
_{0}\!\ext^1_B(I_B/I_B^2,I_{A/B}),$$ and $ codi = c(I_{A/B})+\
_{0}h^{2}(R,A,A)$ if and only if $X$ is unobstructed.
\end{proposition}

Here ``$(M,B)$ unobstructed along any graded deformation of $B$'' means that
for every graded deformation $(M_S,B_S)$ of $(M,B)$, $S$ local and Artinian
with residue field $k$, there is a graded deformation of $M_S$ to any graded
deformation $B_T$ of $B_S$ for any small Artin surjection $T \rightarrow S$
(\cite{K}; Definition 11). The important remark for our application $M=N_B$
where the codimension $2$ CM quotient $B$  satisfies $\depth_{I(Z)}B
\geq 4$, is that {\it all} assumptions of the proposition are satisfied
provided $char(k)=0$ and $(B \rightarrow A)$ is general (see proof of
Corollary 41 and Remark 42 of \cite{K}).

Moreover recall that if we put $\delta:=\ -\delta(I_{A/B})_{0} = \
_{0}\!\ext^1_B(I_B/I_B^2,I_{A/B})- \ _{0}\!\hom_R(I_B,I_{A/B})$ and we use the
exact sequence \eqref{Ns},
we get $\delta=\ \delta(K_B)_{t-2s} -\delta(N_B)_{-s} \ ,$ as previously.

\begin{remark} \label{introproprem} a) With assumptions as in
  Proposition~\ref{introcodimprop}, except for $(B \rightarrow A)$ being
  general we can also show

  \hspace{4 cm} $\delta - \ _{t}\!\ext_B^2(S^2(I_{A/B}(s)),K_B) \le codi
  \ . $ \\[0.2cm]
  b) Moreover if $\depth_{I(Z)}B \geq 4$, then we show
  $$\ _{t}\!\Ext_B^2(S^2(I_{A/B}(s)),K_B) \simeq \ _{s}\!\Ext_B^1(M,A)$$
  exactly as we did for $M=N_B$ in the proof of Corollary~\ref{introcorNB}.
  Hence if $ _{s}\!\Ext_B^1(M,A)=0$, then the lower bound $ c(I_{A/B})$ of
  Proposition~\ref{introcodimprop} is equal to the upper bound and we essentially
  get Corollary~\ref{introcorNB} (ii)! Moreover since $codi \ge 0$,
  Corollary~\ref{introcorNB} (i) corresponds to the case where the upper bound
  is zero!
\end{remark}

\begin{remark} \label{remcor2NB} In the case $s> \max n_{2,j}/2$,
  $\depth_{I(Z)}B \geq 4$ and $char(k)=0$, the inequalities of
  Proposition~\ref{introcodimprop} lead to
  $$
  \epsilon + \delta - \ _0\!\ext_B^1(N_B,A) \leq \ \dim_{(X)}
  \Hi^{p(x)}(\PP^{n}) \leq \epsilon + \delta \
  $$  with $ \epsilon$ as in Corollary~\ref{introcorNB}.
 \end{remark}

 \begin{example} \label{example3} Now let $X=\Proj(A)\subset \PP^5$ be a
   general arithmetically Gorenstein curve defined by the submaximal minors of
   a $3\times 3$ matrix whose first (resp. second) row consists of linear
   (resp. quadratic) forms and whose last row are forms of degree $s-3$ ($s\ge
   5$), i.e. $b_i=0$ for $1\le i\le 3$, $a_1=1, a_2=2$ and $a_3=s-3$. In the
   following we skip a few details which we leave to the reader. Note that the
   case $a_t > a_{t-1}+ a_{t-2}-b_1$ or equivalently $s>6$, is taken care of
   by Theorem \ref{mainthm1}. So we concentrate on the cases $s=5$ and 6,
   which we analyse by using Proposition~\ref{introcodimprop} and
   Remark~\ref{introproprem}. First, we use Remark
   \ref{remcorNB}
   to compute $\ _0\! \ext^2_B(N_B,N_B)$
where $B$ is obtained by deleting the last row and taking maximal
minors. We easily get $\ _0\! \ext^2_B(N_B,N_B) =\  _s\!
\ext^1_B(N_B,A)=3$ by using $$ 0\longrightarrow R(-5) \oplus R(-4)
  \longrightarrow R(-3)^3 \longrightarrow R \longrightarrow B \longrightarrow
  0,
$$
\eqref{I2codim2} and $\ _0\!\Ext_B^2(N_B,N_B)\ \simeq \
_0\!\Hom_B(I_B/I_B^2, \HH^4_{\mathfrak m}(I_B^2))$. Moreover
$\dim(K_B)_{6-2s}=0$ by \eqref{Kcodim2}. Now if we apply $ \
_{-2s}\!\Hom(-,K_B(6))$ to (\ref{alghom}) we get
$\delta(K_B)_{6-2s}=0$ and $ _{-2s}\!\Ext_B(I_B/I_B^2,K_B(6))=0$
for $s \ge 5$ provided we can show $ \
_{-2s}\!\Hom(H_1,K_B(6))=0$. Using (\ref{alghom}) we get that
$H_1$ has rank 1 and $H_1 \simeq K_B(-3)$. Hence $ \
_{-2s}\!\Hom(H_1,K_B(6)) \simeq B(9)_{-2s}=0$ for $s \ge5$.

It remains to compute $\delta(N_B)_{-s}$. We claim that $\delta(N_B)_{-s}=-8$
(resp. $\delta(N_B)_{-s}=-3$) for $s=5$ (resp. $s=6$). Indeed dualizing the
exact sequence (\ref{alghom}) we get
$$ 0 \longrightarrow N_B\longrightarrow B(3)^3\longrightarrow
H_1^*\longrightarrow 0 \ .$$ If we apply $\
_{-s}\!\Hom(I_B/I_B^2,-)$ to this sequence, recalling $H_1 \simeq
K_B(-3)$ and hence $ \ _{-s}\!\Hom(I_B/I_B^2,H_1^*)\simeq \
(I_B/I_B^2)_{9-s}$, we get an exact sequence which rather easily
proves the claim. It follows that the numbers $\delta =
\delta(K_B)_{6-2s}-\delta(N_B)_{-s}$ and $\ _s\! \ext^1_B(N_B,A)$
appearing in Remark~\ref{introproprem} are computed. We conclude,
for $s=5$, that the codimension, $codi$, of the stratum in $\Hi
^{p(x)}(\PP^{5})$ of subschemes given by (\ref{introMs}) around
$(X)$ is at least 5-dimensional. In fact a Macaulay 2 computation
shows $_{0}h^{2}(R,A,A)=0$ and $_0\!\hom(I_B,I_{A/B})=1$ and hence
we have $ codi = c(I_{A/B})+\ _{0}h^{2}(R,A,A) = 6$ by
Proposition~\ref{introcodimprop}. For $s=6$ the lower bound for
$codi$ of Remark~\ref{introproprem} (a) is $0$. Since a Macaulay 2
computation shows $_0\!\hom(I_B,I_{A/B})=0$ the better lower bound
of Proposition~\ref{introcodimprop} is also 0 while the smallest
upper bound of Proposition~\ref{introcodimprop} is 3. The latter
is the correct bound for the codimension of the stratum provided
$X$ is unobstructed. In conclusion if $X$ belongs to a reduced
component $V$ of $\Hi^{p(x)}(\PP^{5})$, then $codi=3$, but
$codi=0$ is possible in which case $V$ is non-reduced. We have not
been able to fully tell what happens, but we expect $V$ to be
reduced and $codi=3$.
\end{example}

The last case of the preceding example illustrates how difficult the analysis
of when $codi$ is positive could be. Especially cases where $a_t$ is close to
$a_{t-1}+ a_{t-2}-b_1$ seem difficult to handle. Since it turns out that the
lower bounds of Proposition~\ref{introcodimprop} and Remark~\ref{introproprem}
(a) are often negative (also in the case $a_t > a_{t-1}+ a_{t-2}-b_1$ treated
in Theorem \ref{mainthm1}), they are not very helpful. This, however, also
indicates that the conclusions of Theorem \ref{mainthm1} are rather strong.

\end{document}